\documentclass[12pt,a4paper]{amsart}
\usepackage{amsmath,amsthm,amssymb,latexsym,a4wide}

\newtheorem{theorem}{Theorem}
\newtheorem{lemma}[theorem]{Lemma}
\newtheorem{corollary}[theorem]{Corollary}
\newtheorem{proposition}[theorem]{Proposition}

\newtheorem{remark}[theorem]{Remark}
\usepackage[all]{xy}
\usepackage{enumerate}
\numberwithin{equation}{section}

\begin{document}

\title{Partitioned binary relations}
\author{Paul Martin and Volodymyr Mazorchuk}

\begin{abstract}
We define the category of partitioned binary relations and show
that it contains many classical diagram categories, including 
categories of binary relations, maps, injective maps,
partitions, (oriented) Brauer diagrams and (oriented)
Temperley-Lieb diagrams. We construct a one-parameter deformation 
of the  category of partitioned binary relations and show that it
gives rise to classical one-parameter deformations of 
partition, Brauer and Temperley-Lieb categories.
Finally, we describe a factorization of partitioned binary
relations into a product of certain idempotents and pairs of usual
binary relations.
\end{abstract}
\maketitle

\section{Introduction and description of the results}\label{s1}

Diagram algebras and categories are interesting and rich objects 
of study in modern representation theory with many application to,
among others, statistical mechanics, see the book \cite{Mar2}
and the surveys \cite{Mar3,Koe}, and topology, see \cite{RT}. 
Classical diagram categories include the Brauer category (see 
\cite{Br}), the partition category (see \cite{Mar1}), the 
Temperley-Lieb category (which has many important 
applications  in topology, combinatorics and categorification,
see e.g. \cite{TL,BFK}) and their partial 
(alias rook) analogues (see \cite{Maz1,Gr,HL}), together with 
the category of binary relations (confer e.g. \cite{PW}).   
From the algebraic perspective all these categories have rich and 
non-trivial structure, though much less is known for the category
of binary relations than the others. Morphisms in these categories 
are described in terms of certain combinatorially defined sets with 
diagrammatic realization. Furthermore, most of the classical diagram 
categories admit a non-trivial one-parameter deformation, which also plays
a very important role in certain applications (see e.g. \cite{Br}).

The aim of the present paper is to show that both the partition
category and the category of binary relations are shadows of a
more general natural construction. We define a new category
which we call category of {\em partitioned binary relations}
and show that it provides a single overarching setting for 
all the categories mentioned above. Our main results are:
\begin{itemize}
\item The well-definedness of the new category (Theorem~\ref{thm99}).
\item Connection of the new category with the above mentioned 
classical objects (Section~\ref{s3}).
\item Functorial comparison of the representation theories of the new
category and the category of binary relations (Subsection~\ref{s3.1}).
\item Well-definedness of a certain flat deformation
(Theorem~\ref{thm5}), which has application 
in representation theory (confer \cite{CPS,CMPX}).
\item Factorization of morphisms in the new category in terms of
simpler structures (Theorem~\ref{thm77}).
\end{itemize}

A notable feature of our construction is that it is not 
straightforward. An obvious approach to such an overarching
construction is to relax the reflexive-symmetric-transitive condition
on the relations that constitute morphisms in the partition category.
In fact, this does yield the morphisms in the new category, 
but it does not determine a composition.
Another indication comes from the Temperley-Lieb category, or rather
its (topologically motivated) ``oriented'' generalization 
(see e.g.  \cite{Tu}). This is easy to extend to the level of the
Brauer category and the corresponding partial analogues. The diagrams
of this oriented version can be viewed as oriented graphs, which 
suggests a connection to the category of binary relations.
It is worth pointing out that both the partition category and the
category of binary relations have also 
recently appeared in a different context in \cite{FW}. 

The category of binary relations, or rather its endomorphism monoids, are 
classical objects of study in semigroup theory, see \cite{PW,Sc,Ko} 
and references therein. In \cite{MP} it is shown that every finite group 
appears as a maximal subgroup of some monoid of binary relations,
which shows that these monoids are structurally more complicated than 
the classical transformation semigroups generalizing the symmetric group 
(see \cite{GM} for the latter). 

The paper is organized as follows: In Section~\ref{s2} we define
the category $\mathfrak{PB}$ of partitioned binary relations;
in Section~\ref{s3} we show that it contains many classical 
categories mentioned above; in Section~\ref{s4} we show that 
the category $\mathfrak{PB}$ has a flat one-parameter deformation.
In Section~\ref{s5} we describe a factorization of 
partitioned binary relations, which we call {\em polarized
factorization}. It turns our that every partitioned binary relation
can be written as a product of three elements, two of which are
idempotents of a certain simple form, and the third one is a
``pair'' of usual binary relations. As an application, we show that 
almost all products of partitioned binary relations result in 
the full partitioned binary relation (in the limit of ``large'' objects).
\vspace{2mm}

\noindent
{\bf Acknowledgements.} An essential part of the research was done
during the visit of the first author to Uppsala, which was supported
by the Faculty of Natural Sciences of Uppsala University. 
The financial support and hospitality of Uppsala University are
gratefully acknowledged. For the second author the research was
partially supported by the Swedish Research Council.

\section{Category of partitioned binary relations}\label{s2}

We denote by $\mathbb{N}$ and $\mathbb{N}_0$ the sets of all
positive and non-negative integers, respectively.

\subsection{Partitioned binary relations}\label{s2.1}

Let $X$ and $Y$ be finite sets. A {\em partitioned binary relation} 
(PBR) on $(X,Y)$ is a binary relation $\alpha$ on the 
{\em disjoint} union of $X$  and $Y$. The sets $X$ and $Y$ are called 
the {\em domain} and the {\em codomain} of $\alpha$ and denoted 
by $\mathrm{Dom}(\alpha)$ and $\mathrm{Codom}(\alpha)$, respectively.
Clearly, the number of partitioned binary relations on $(X,Y)$
equals $2^{(|X|+|Y|)^2}$.

Sometimes it might happen that $X\cap Y\neq \varnothing$, or even
$X=Y$. In this case to distinguish between elements of the domain and 
the codomain, we write $a^{(d)}$ or $a^{(c)}$ for elements of 
$\mathrm{Dom}(\alpha)$ and $\mathrm{Codom}(\alpha)$, respectively.

A PBR $\alpha$ on $(X,Y)$ will be depicted as a directed graph drawn 
within a rectangular frame, with elements of $X$ and $Y$ 
represented by vertexes positioned on the right and left hand sides of 
the frame, respectively. The fact that $\alpha$ contains an edge 
$(a,b)\in (X\coprod Y)^2$ will be written $(a,b)\in\alpha$ and 
visualized by an arrow from $a$ to $b$ on the graph.
We will call $a$ and $b$ {\em elements} while $(a,b)$ will be called
an {\em edge}. An example of a partitioned binary relation from
$X=\{x_1,x_2,x_3,x_4\}$ to $Y=\{y_1,y_2,y_3,y_4,y_5,y_6,y_7,y_8,y_9\}$
is shown in Figure~\ref{fig1}. One can and we will use diagrams 
interchangeably with the set theoretic approach to PBRs.

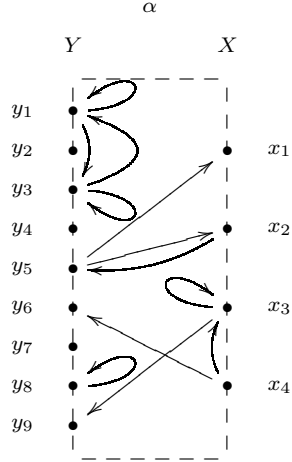
\begin{figure}
{\tiny
\begin{displaymath}
\xymatrixrowsep{1mm}
\xymatrixcolsep{1.5mm}
\xymatrix{
&&&\alpha&&&\\
&Y&&&&X&\\
&\ar@{--}[dddddddddd]\ar@{--}[rrrr]&&&&\ar@{--}[dddddddddd]&\\
y_1&\bullet\ar@(r,ur)[]\ar@/^/[dd]&&&&&\\
y_2&\bullet&&&&\bullet&x_1\\
y_3&\bullet\ar@(r,dr)[]\ar@/_2pc/[uu]&&&&&\\
y_4&\bullet&&&&\bullet\ar@/^/[lllld]&x_2\\
y_5&\bullet\ar[rrrruuu]\ar[rrrru]&&&&&\\
y_6&\bullet&&&&\bullet\ar[llllddd]\ar@(l,ul)[]&x_3\\
y_7&\bullet&&&&&\\
y_8&\bullet\ar@(r,ur)[]&&&&\bullet\ar@/^/[uu]\ar[uullll]&x_4\\
y_9&\bullet&&&&&\\
&\ar@{--}[rrrr]&&&&&
}
\end{displaymath}
\caption{A partitioned binary relation on $(X,Y)$}
\label{fig1}
}
\end{figure}

\subsection{Composition of partitioned binary relations}\label{s2.2}

In this subsection we define composition of PBRs in a categorical 
sense. That it, given a PBR $\alpha$ on $(X,Y)$ and a
PBR $\beta$ on $(Y,Z)$, we define their composition 
$\beta\circ\alpha$, which will be a PBR on $(X,Z)$. 

It will be convenient to start slightly more generally. Let 
$\aleph=(\alpha_1,\alpha_2,\alpha_3,\dots,\alpha_k)$ 
be a {\em composable}
sequence of PBRs in the above sense, 
that is $\mathrm{Codom}(\alpha_i)=
\mathrm{Dom}(\alpha_{i+1})$ for all $i=1,2,\dots,k-1$. 
Set $X_i:=\mathrm{Dom}(\alpha_{i})$ for $i=1,2,\dots,k$,
$X_{k+1}:=\mathrm{Codom}(\alpha_k)$, and 
$\displaystyle
X_{\coprod} :=\coprod_{i=1}^{k+1} X_i$.
A sequence $\xi=(a_1,b_1),(a_2,b_2),\dots,(a_m,b_m)$ of edges taken
from the PBRs in $\aleph$ is called {\em $\aleph$-connected} provided that 
\begin{enumerate}[(I)]
\item\label{c1} no two successive edges in $\xi$ are in the same PBR;
\item\label{c2} for every $i=1,2,\dots,m-1$ we have  $b_i=a_{i+1}$ 
(as elements of $X_{\coprod}$).
\end{enumerate}
We will also say that the $\aleph$-connected sequence 
$\xi$ {\em connects} $a_1$ to $b_m$. Note that on every step $i$
the element $b_i$ defines the PBR $\alpha_j$ containing $(a_{i+1},b_{i+1})$
uniquely due to condition \eqref{c2}. Note also that in the case
$k=1$, we necessarily have $m=1$.

Let $\alpha$ be a PBR on $(X,Y)$ and $\beta$ be a PBR
on $(Y,Z)$. We define the {\em composition} $\beta\circ\alpha$
as the PBR on $(X,Z)$ such that for every $a,b\in X\coprod Z$
the PBR $\beta\circ\alpha$ contains $(a,b)$ if and only if there
exists an $(\alpha,\beta)$-connected sequence connecting $a$ to $b$.
An example of composition of two PBRs is shown in Figure~\ref{fig2}.

\begin{figure}
{\tiny
\begin{displaymath}
\xymatrixrowsep{1mm}
\xymatrixcolsep{1.5mm}
\xymatrix{
&&\beta&&&&&\alpha&&&&&&&\beta\circ\alpha&&\\
\ar@{--}[rrrr]\ar@{--}[dddddddddd]&&&&\ar@{--}[dddddddddd]
&\ar@{--}[rrrr]\ar@{--}[dddddddddd]&&&&\ar@{--}[dddddddddd]&&&
\ar@{--}[rrrr]\ar@{--}[dddddddddd]&&&&\ar@{--}[dddddddddd]\\
&&&&\bullet\ar@{.}[r]&\bullet\ar@(r,ur)[]&&&&&&&&&&&\\
\bullet\ar@/^/[rrrr]&&&&\bullet\ar@/^/[llll]\ar@{.}[r]&\bullet\ar@/^/[rrrr]
\ar@/^2pc/[d]&&&&\ar@/^1pc/[llll]\bullet&&&
\bullet\ar@/^/[rrrr]\ar@/^/[rrrrd]\ar@/^/[dd]&&&&\bullet\ar@/^/[llll]\\
&&&&\bullet\ar@{.}[r]\ar@(l,dl)[]&\bullet\ar@/^2pc/[d]
&&&&\bullet&&&&&&&\bullet\\
\bullet&&&&\bullet\ar@{.}[r]\ar@/_1pc/[d]&
\bullet&&&&\bullet\ar@/_1.5pc/[d]&&&\bullet&&&&
\bullet\ar@/_1.5pc/[d]\\
&&&&\bullet\ar@{.}[r]\ar@/^/[llllu]\ar@/^2.5pc/[u]
&\bullet\ar[rrrruu]\ar@(r,dr)[]&&&&\bullet\ar@(l,dl)[]
&&&&&&&\bullet\ar@(l,dl)[]\\
\bullet\ar@/_/[rrrrdd]&&&&
\bullet\ar@{.}[r]\ar@(l,dl)[]&\bullet\ar[rrrrd]
&&&&\bullet&&&\bullet&&&&\bullet\\
&&&&\bullet\ar@{.}[r]&\bullet\ar[rrrrd]
&&&&\bullet\ar@/^/[llllu]&&&&&&&\bullet\ar@(l,dl)[]\\
\bullet\ar[rrrrd]&&&&\bullet\ar@{.}[r]\ar@/_/[lllluu]&
\bullet\ar@/^2.5pc/[d]&&&&\bullet\ar@/^/[llllu]&&&
\bullet\ar@/_/[uu]&&&&\bullet\\
&&&&\bullet\ar@{.}[r]&\bullet\ar@/_1pc/[u]&&&&&&&&&&&\\
\ar@{--}[rrrr]&&&&&\ar@{--}[rrrr]&&&&&&&\ar@{--}[rrrr]&&&&\\
}
\end{displaymath}
\caption{Composition of partitioned binary relations}
\label{fig2}
}
\end{figure}
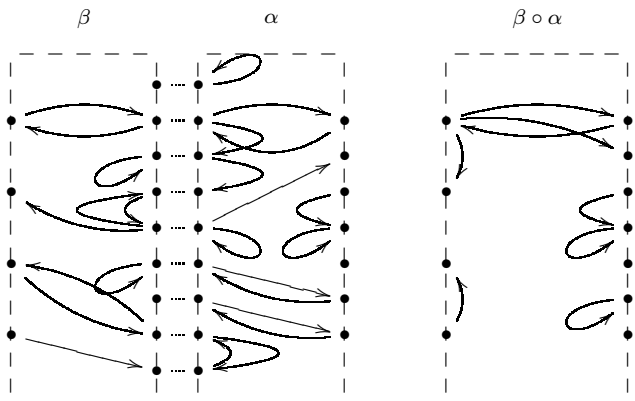

\subsection{Category of partitioned binary relations}\label{s2.3}

A principal observation is the following:

\begin{proposition}\label{thm1}
Composition $\circ$ defined above is associative.
\end{proposition}

\begin{proof}
Let $\alpha$ be a PBR on $(X,Y)$, $\beta$ be a PBR on $(Y,Z)$,
and $\gamma$ be a PBR on $(Z,U)$. Set $\aleph:=(\alpha,\beta,\gamma)$,
$\xi:=\beta\circ\alpha$
and $\zeta:=\gamma\circ\beta$. To prove our theorem we have to check
that $(a,b)\in \gamma\circ\xi$ implies $(a,b)\in \zeta\circ\alpha$
for every $(a,b)\in (X\coprod Z)^2$ and vice versa. We prove the 
first claim, the second one is proved similarly.

Let $(a,b)\in \gamma\circ\xi$ for some $(a,b)\in (X\coprod Z)^2$.
Then there is a $(\xi,\gamma)$-connected sequence 
$(a_1,b_1),(a_2,b_2),\cdots,(a_k,b_k)$ connecting
$a$ to $b$. From this $(\xi,\gamma)$-connected sequence create a new
sequence of edges by replacing every edge $(a_i,b_i)\in\xi$ in this sequence 
by an $(\alpha,\beta)$-connected sequence connecting $a_i$ to $b_i$
(such a sequence exists by definition of composition, but it is not 
necessarily unique). By construction, the obtained sequence will be
$\aleph$-connected.

Consider now all maximal consecutive subsequences of this sequence,
containing only edges from $\beta$ and $\gamma$. By maximality,
each such subsequences is both preceded and followed by an edge 
from $\alpha$, if any. From the $\aleph$-connectedness of the
original sequence it follows that
any such subsequence is a $(\beta,\gamma)$-connected sequence 
connecting its first element to its last element. 
Construct a new sequence by replacing each such
maximal $(\beta,\gamma)$-connected subsequence by the
pair of elements which this subsequence connects. 
This pair of elements gives an edge in $\zeta$ by definition. As a result, 
we obtain an $(\alpha,\zeta)$-connected sequence connecting $a$ to $b$.
Hence $(a,b)\in\zeta\circ\alpha$. The claim follows.
\end{proof}

For a finite set $X$ define the PBR $\varepsilon_X$ on $(X,X)$
as the one containing all edges $(x^{(d)},x^{(c)})$ and $(x^{(c)},x^{(d)})$
for all $x\in X$. The diagram of the PBR $\varepsilon_X$ is shown 
in Figure~\ref{fig3}.

\begin{figure}
{\tiny
\begin{displaymath}
\xymatrixrowsep{3mm}
\xymatrixcolsep{1.5mm}
\xymatrix{
&&\varepsilon_X&&&&&&&&\overline{\varepsilon}_X&&
&&&&&&\hat{\varepsilon}_X&&\\
X&&&&X&&&&X&&&&X&&&&X&&&&X\\
\ar@{--}[dddddd]\ar@{--}[rrrr]&&&&\ar@{--}[dddddd]
&&&&\ar@{--}[rrrr]\ar@{--}[dddddd]&&&&\ar@{--}[dddddd]
&&&&\ar@{--}[rrrr]\ar@{--}[dddddd]&&&&\ar@{--}[dddddd]\\
\bullet\ar@/^/[rrrr]&&&&\bullet\ar@/^/[llll]&&&&
\bullet\ar@/^/[rrrr]\ar@(r,ur)[]&&&&\bullet\ar@/^/[llll]\ar@(l,dl)[]
&&&&\bullet&&&&\bullet\ar[llll]\\
\bullet\ar@/^/[rrrr]&&&&\bullet\ar@/^/[llll]&&&&
\bullet\ar@/^/[rrrr]\ar@(r,ur)[]&&&&\bullet\ar@/^/[llll]\ar@(l,dl)[]
&&&&\bullet&&&&\bullet\ar[llll]\\
&&\dots&&&&&&&&\dots&&&&&&&&\dots\\
\bullet\ar@/^/[rrrr]&&&&\bullet\ar@/^/[llll]&&&&
\bullet\ar@/^/[rrrr]\ar@(r,ur)[]&&&&\bullet\ar@/^/[llll]\ar@(l,dl)[]
&&&&\bullet&&&&\bullet\ar[llll]\\
\bullet\ar@/^/[rrrr]&&&&\bullet\ar@/^/[llll]&&&&
\bullet\ar@/^/[rrrr]\ar@(r,ur)[]&&&&\bullet\ar@/^/[llll]\ar@(l,dl)[]
&&&&\bullet&&&&\bullet\ar[llll]\\
\ar@{--}[rrrr]&&&&&&&&\ar@{--}[rrrr]&&&&&&&&\ar@{--}[rrrr]&&&&
}
\end{displaymath}
\caption{The partitioned binary relations $\varepsilon_X$,
$\overline{\varepsilon}_X$ and $\hat{\varepsilon}_X$}
\label{fig3}
}
\end{figure}
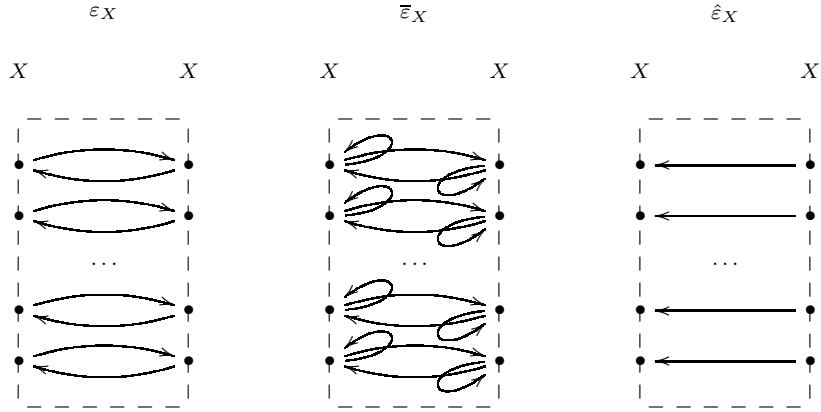

\begin{proposition}\label{prop2}
The PBR $\varepsilon_X$ is the identity morphism for $X$ with respect to 
$\circ$, that is $\varepsilon_X\circ \alpha=\alpha$ for any
PBR $\alpha$ on $(Y,X)$, and $\beta\circ \varepsilon_X=\beta$ 
for any PBR $\beta$ on $(X,Y)$.
\end{proposition}

\begin{proof}
This is a straightforward computation.
\end{proof}

Adding all loops to the PBR $\varepsilon_X$ one obtains the 
idempotent PBR $\overline{\varepsilon}_X$ (see Figure~\ref{fig3}).
Deleting all right arrows from the PBR $\varepsilon_X$ one obtains the 
idempotent PBR $\hat{\varepsilon}_X$ (see Figure~\ref{fig3}).
The PBRs $\overline{\varepsilon}_X$ and $\hat{\varepsilon}_X$ will appear 
as identity morphisms for certain categorical substructures later on.

Define the category $\mathfrak{PB}$ of partitioned binary relations
in the following way. Firstly: objects of $\mathfrak{PB}$ are finite sets; 
for $X,Y\in \mathfrak{PB}$ the morphism set $\mathfrak{PB}(X,Y)$ is 
the set of all PBRs on $(X,Y)$; the composition 
$\mathfrak{PB}(Y,Z)\times \mathfrak{PB}(X,Y)\to
\mathfrak{PB}(X,Z)$ is given by $\circ$; for 
$X\in \mathfrak{PB}$ the identity morphism for $X$
is $\varepsilon_X$. Then, from 
Propositions~\ref{thm1} and \ref{prop2} we obtain:

\begin{theorem}\label{thm99}
The construct $\mathfrak{PB}$ above is a category. 
\end{theorem}

\subsection{Tensor product and duality}\label{s2.5}

The category $\mathfrak{PB}$ has a natural monoidal structure 
in which the tensor product is given on objects by the disjoint 
union and on morphisms by drawing diagrams next to each other as 
shown in Figure~\ref{fig4}.

\begin{figure}
{\tiny
\begin{displaymath}
\xymatrixrowsep{1mm}
\xymatrixcolsep{1.5mm}
\xymatrix{
\ar@{--}[rrrr]\ar@{--}[dddd]&&&&\ar@{--}[dddd]
&\qquad\qquad&&&&&\\
\bullet\ar@/_/[rrrrd]&&&&\bullet\ar@/_2pc/[dd]
&&\ar@{--}[rrrr]\ar@{--}[ddddddddd]&&&&\ar@{--}[ddddddddd]\\
&&&&\bullet\ar@/_/[llllu]&&\bullet\ar@/_/[rrrrd]&&&&\bullet\ar@/_2pc/[dd]\\
\bullet\ar@(r,dr)[]&&&&\bullet&&&&&&\bullet\ar@/_/[llllu]\\
\ar@{--}[rrrr]&&&&&&\bullet\ar@(r,dr)[]&&&&\bullet\\
&&\otimes&&&=&\ar@{--}[rrrr]&&&&\\
\ar@{--}[rrrr]\ar@{--}[ddddd]&&&&\ar@{--}[ddddd]
&&\bullet\ar[rrrr]&&&&\bullet\ar[lllld]\\
\bullet\ar[rrrr]&&&&\bullet\ar[lllld]&&\bullet\ar[rrrr]&&&&
\bullet\ar@/_2pc/[d]\\
\bullet\ar[rrrr]&&&&\bullet\ar@/_2pc/[d]&&\bullet\ar[rrrruu]&&&&\bullet\\
\bullet\ar[rrrruu]&&&&\bullet&&\bullet&&&&\bullet\ar@(l,dl)[]\\
\bullet&&&&\bullet\ar@(l,dl)[]&&\ar@{--}[rrrr]&&&&\\
\ar@{--}[rrrr]&&&&&&&&&&\\
}
\end{displaymath}
\caption{Tensor product}
\label{fig4}
}
\end{figure}
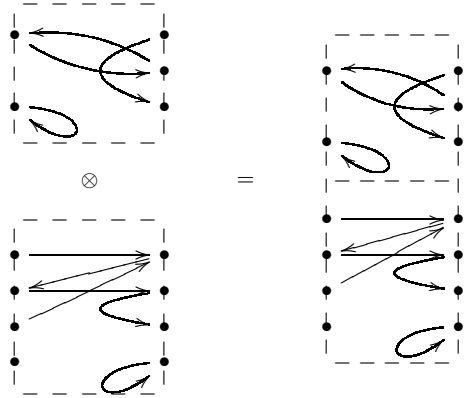

The category $\mathfrak{PB}$ has a natural involution
(that is a contravariant object preserving anti-automorphism), which we
will denote by $\star$, given by taking the mirror image of the 
diagram with respect to a vertical mirror as shown in Figure~\ref{fig5}.

\begin{figure}
{\tiny
\begin{displaymath}
\xymatrixrowsep{1mm}
\xymatrixcolsep{1.5mm}
\xymatrix{
&&\alpha&&&\qquad&&&\alpha^{\star}&&\\
\ar@{--}[dddddddddd]\ar@{--}[rrrr]&&&&\ar@{--}[dddddddddd]&
&\ar@{--}[dddddddddd]\ar@{--}[rrrr]&&&&\ar@{--}[dddddddddd]\\
\bullet\ar@(r,ur)[]\ar@/^/[dd]&&&&&&&&&&\bullet\ar@(l,ul)[]\ar@/_/[dd]\\
\bullet&&&&\bullet&&\bullet&&&&\bullet\\
\bullet\ar@(r,dr)[]\ar@/_2pc/[uu]&&&&&&&&&&\bullet\ar@(l,dl)[]\ar@/^2pc/[uu]\\
\bullet&&&&\bullet\ar@/^/[lllld]&&\bullet\ar@/^/[rrrrd]&&&&\bullet\\
\bullet\ar[rrrruuu]\ar[rrrru]&&&&&&&&&&\bullet\ar[lllluuu]\ar[llllu]\\
\bullet&&&&\bullet\ar[llllddd]\ar@(l,ul)[]&&
\bullet\ar[rrrrddd]\ar@(r,ur)[]&&&&\bullet\\
\bullet&&&&&&&&&&\bullet\\
\bullet\ar@(r,ur)[]&&&&\bullet\ar@/^/[uu]\ar[uullll]&&
\bullet\ar@/_/[uu]\ar[uurrrr]&&&&\bullet\ar@(l,ul)[]\\
\bullet&&&&&&&&&&\bullet\\
\ar@{--}[rrrr]&&&&&&\ar@{--}[rrrr]&&&&
}
\end{displaymath}
\caption{Anti-automorphism $\star$}
\label{fig5}
}
\end{figure}
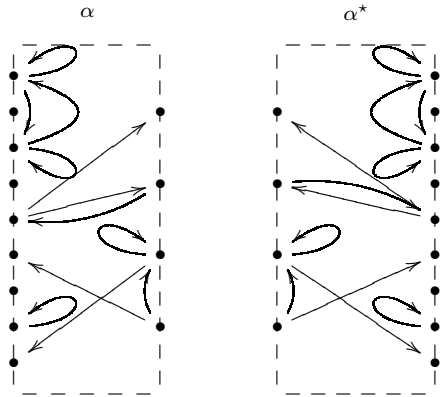

\section{Some substructures of $\mathfrak{PB}$}\label{s3}

\subsection{Binary relations, first inclusion}\label{s3.1}

Consider the category $\mathfrak{B}$ of binary relations between
finite sets (confer \cite{PW}). Objects of $\mathfrak{B}$ are finite sets. For
$X,Y\in \mathfrak{B}$, the set $\mathfrak{B}(X,Y)$ is the set of
all binary relations from $X$ to $Y$. A binary relation from
$X$ to $Y$ is a subset of $X\times Y$. Such a binary relation can
be viewed as a boolean matrix whose columns are indexed by elements
of $X$ and rows are indexed by elements of $Y$. We shall treat the 
two realizations as interchangeable. Composition of binary relations 
may then be lifted from the usual boolean multiplication of 
boolean  matrices (see, e.g. \cite{PW}). The identity morphism for 
$X$ is the equality relation (it is given by the 
identity matrix with respect to the same ordering of the two copies
of $X$). The category $\mathfrak{B}$ has a natural
involution $\bowtie$ given by matrix transposition.

Each binary relation from $X$ to $Y$ is a partitioned binary
relation from $X$ to $Y$, in other words, 
$\mathfrak{B}(X,Y)\subset \mathfrak{PB}(X,Y)$. It is 
straightforward to check that this inclusion respects composition. 
We will denote this inclusion by $\Phi_1$. Note that 
$\Phi_1$ is {\em not} a functor as it does not 
send the identity binary relation to the identity partitioned binary
relation. 

Note that $\mathfrak{B}$ has several classical subcategories, in particular,
\begin{enumerate}[(i)]
\item\label{ex01}  the subcategory of all maps;
\item\label{ex02}  the subcategory of all injective maps;
\item\label{ex03}  the subcategory of all partial injective maps;
\item\label{ex04}  the subcategory of all surjective maps;
\item\label{ex05}  the subcategory of all partial surjective maps.
\end{enumerate}
We refer the reader to \cite{KM} for details on categories
\eqref{ex03} and \eqref{ex05}. Using $\Phi_1$ we obtain 
inclusions of all these  categories into $\mathfrak{PB}$ by restriction.

The image $\Phi_1(\mathfrak{B})$ can also be understood as an 
{\em idempotent} subcategory  of  $\mathfrak{PB}$.  Let $\mathcal{C}$
be a category and $\mathbf{e}=(e_X)_{X\in\mathcal{C}}$ a fixed
collection of idempotent endomorphisms such that $e_X\in\mathcal{C}(X,X)$.
An {\em $\mathbf{e}$-subcategory} $\mathcal{D}$
of $\mathcal{C}$ is a category such that
\begin{itemize}
\item objects of $\mathcal{D}$ form a subclass
of objects of $\mathcal{C}$;
\item $\mathcal{D}(X,Y)\subset \mathcal{C}(X,Y)$ for any 
$X,Y\in\mathcal{D}$;
\item the multiplication in $\mathcal{D}$ is obtained from the one
in $\mathcal{C}$ by restriction;
\item for any $X\in\mathcal{D}$ the
morphism $e_X$ is the corresponding identity morphism for $X$.
\end{itemize}
Among all $\mathbf{e}$-subcategories of  $\mathcal{C}$ there is the 
unique maximum  one with respect to inclusions.
This category is denoted by $\mathcal{C}_{\mathbf{e}}$, it has the
same objects as $\mathcal{C}$ and for $X,Y\in \mathcal{C}$ we have
\begin{displaymath}
\mathcal{C}_{\mathbf{e}}(X,Y)=e_Y \mathcal{C}(X,Y)e_X.
\end{displaymath}

\begin{remark}\label{rem123}
{\rm 
Similarly to \cite[Section~5]{Au} one shows that
the category of $\mathcal{C}_{\mathbf{e}}$-representations
over some field $\Bbbk$ (that is functors from 
$\mathcal{C}_{\mathbf{e}}$ to $\Bbbk$-vector spaces) fully embeds 
into the category of $\mathcal{C}$-representations. 
} 
\end{remark}

For $X\in \mathfrak{B}$ recall the idempotent
PBR  $\hat{\varepsilon}_X$ defined in Subsection~\ref{s2.3}
(see Figure~\ref{fig3}). The PBR $\hat{\varepsilon}_X$ is the image of
the identity relation on $X$ under $\Phi_1$. 

\begin{proposition}\label{prop61}
For $\mathbf{e}:=(\hat{\varepsilon}_X)_{X\in \mathfrak{B}}$
we have $\Phi_1(\mathfrak{B})=\mathfrak{PB}_{\mathbf{e}}$.
\end{proposition}

\begin{proof}
We have to check that for any $X,Y\in \mathfrak{PB}$ and 
$\alpha\in \mathfrak{PB}(X,Y)$ the following is true:
$\alpha\in \Phi_1(\mathfrak{B})$ if and only if 
$\alpha=\hat{\varepsilon}_Y\circ\beta\circ\hat{\varepsilon}_X$
for some $\beta\in \mathfrak{PB}(X,Y)$. 

If $\alpha\in \Phi_1(\mathfrak{B})$, then 
$\hat{\varepsilon}_Y\circ\alpha\circ\hat{\varepsilon}_X=\alpha$.
On the other hand, let $\beta\in \mathfrak{PB}(X,Y)$,
$\alpha=\hat{\varepsilon}_Y\circ\beta\circ\hat{\varepsilon}_X$
and $(a,b)$ be an edge of $\alpha$. Let $(a_1,b_1),\dots,(a_m,b_m)$
be an $(\hat{\varepsilon}_X,\beta,\hat{\varepsilon}_Y)$-connected 
sequence connecting $a$ to $b$. From the definition
of $\hat{\varepsilon}_Y$ it follows that $(a_1,b_1)$ is an edge from
$\hat{\varepsilon}_X$. Similarly, $(a_m,b_m)$ is an edge from 
$\hat{\varepsilon}_Y$. This implies $a\in X$ and $b\in Y$.
The claim follows.
\end{proof}

\subsection{Binary relations, second inclusion}\label{s3.2}

With each binary relation $\theta$ from $X$ to $Y$ we associate
a partitioned binary relation $\Phi_2(\theta)$ on $(X,Y)$
in the following way: $\Phi_2(\theta):=\Phi_1(\theta)\cup
\Phi_1(\theta^{\bowtie})$. The effect of $\Phi_2$ on binary relations
is illustrated in Figure~\ref{fig7}. 

\begin{proposition}\label{prop4}
The map $\Phi_2$ gives rise to a faithful functor from
$\mathfrak{B}$ to $\mathfrak{PB}$.
\end{proposition}

\begin{proof}
A proof will be given in Remark~\ref{rem73n}.
\end{proof}

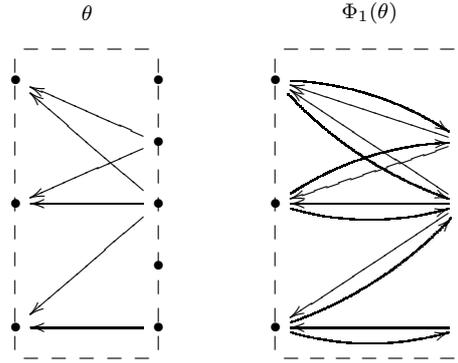
\begin{figure}
{\tiny
\begin{displaymath}
\xymatrixrowsep{1mm}
\xymatrixcolsep{1.5mm}
\xymatrix{
&&\theta&&&\qquad&&&\Phi_1(\theta)&&\\
\ar@{--}[dddddddddd]\ar@{--}[rrrr]&&&&\ar@{--}[dddddddddd]&
&\ar@{--}[dddddddddd]\ar@{--}[rrrr]&&&&\ar@{--}[dddddddddd]\\
\bullet&&&&\bullet&&\bullet\ar@/^/[rrrrdd]\ar@/_/[rrrrdddd]&&&&\bullet\\
&&&&&&&&&&\\
&&&&\bullet\ar[lllluu]\ar[lllldd]&&&&&&
\bullet\ar[lllluu]\ar[lllldd]\\
&&&&&&&&&&\\
\bullet&&&&\bullet\ar[lllluuuu]\ar[lllldddd]\ar[llll]&&
\bullet\ar@/_/[rrrr]\ar@/^/[rrrruu]&&&&
\bullet\ar[llll]\ar[lllluuuu]\ar[lllldddd]\\
&&&&&&&&&&\\
&&&&\bullet&&&&&&\bullet\\
&&&&&&&&&&\\
\bullet&&&&\bullet\ar[llll]&&
\bullet\ar@/_/[rrrr]\ar@/_/[rrrruuuu]&&&&\bullet\ar[llll]\\
\ar@{--}[rrrr]&&&&&&\ar@{--}[rrrr]&&&&
}
\end{displaymath}
\caption{Injection $\Phi_2$}
\label{fig7}
}
\end{figure}

Similarly to Subsection~\ref{s3.1}, using $\Phi_2$ we realize
categories of various types of maps as subcategories of 
$\mathfrak{PB}$.

\subsection{Partition category}\label{s3.3}

Denote by $\mathfrak{P}$ the {\em partition} category, defined as follows
(see \cite{Mar1}): Objects of $\mathfrak{P}$ are finite sets.
For $X,Y\in \mathfrak{P}$ the set $\mathfrak{P}(X,Y)$ is the set of
all partitions of $X\coprod Y$ into a disjoint union of subsets
(called {\em parts}). 
For $\alpha\in \mathfrak{P}(X,Y)$ and $\beta\in \mathfrak{P}(Y,Z)$
the composition $\beta\circ\alpha$ is defined as the unique
partition in $\mathfrak{P}(X,Z)$ such that for any $a,b\in X\coprod Z$
the elements $a$ and $b$ belong to the same part of the partition
$\beta\circ\alpha$ if and only if for some $k\in\mathbb{N}_0$ there is 
a sequence $a=a_0,a_1,\dots,a_k=b$ of elements from $X\coprod Y\coprod Z$
such that for every $i=0,1,\dots,k-1$ the elements $a_i$ and $a_{i+1}$
belong to the same part of either $\alpha$ or $\beta$. The identity
morphism $\pi_X$ of $\mathfrak{P}(X,X)$ is the partition of 
$X\coprod X=X\cup X'$,
where $X':=\{x',x\in X\}$, consisting of parts $\{x,x'\}$, $x\in X$.

For $X\in \mathfrak{P}$ set $\Psi(X)=X\in \mathfrak{PB}$. For
$\alpha\in \mathfrak{P}(X,Y)$ denote by $\Psi(\alpha)$ the unique
PBR in $\mathfrak{PB}(X,Y)$ such that for every $a,b\in X\coprod Y$
we have $(a,b)\in \Psi(\alpha)$ if and only if $a$ and $b$ belong
to the same part of $\alpha$. Alternatively, we can say that the binary
relation $\Psi(\alpha)$ is obtained by considering the partition 
$\alpha$ of $X\coprod Y$ as an equivalence relation on 
$X\coprod Y$. Note that $\Psi(\pi_X)=\overline{\varepsilon}_X$.

A partition is usually drawn as a diagram similarly to a diagram of PBR.
Elements of the diagram are connected such that the connected components
correspond to parts of the partition (note that a diagram of a partition
is not uniquely defined). An example of how $\Psi$ works 
is given in Figure~\ref{fig8} (note the use of double arrows there
to simplify the picture). It is straightforward to verify that
for any $\alpha\in\mathfrak{P}(X,Y)$ and $\beta\in\mathfrak{P}(Y,Z)$
we have  $\Psi(\beta\circ\alpha)=\Psi(\beta)\circ\Psi(\alpha)$.

Note that $\Psi$ is a not a functor as it does not map identity 
morphisms to identity morphisms. The image of
$\Psi$ does not coincide with the idempotent subcategory of 
$\mathfrak{PB}$ generated by $\mathbf{e}=
\{\overline{\varepsilon}_X, X\in \mathfrak{PB}\}$.
The latter idempotent subcategory is larger. One can readily see that 
the subset of reflexive, transitive relations in $\mathfrak{PB}$ is 
closed under composition, and that this is $\mathfrak{PB}_{\mathbf{e}}$.

\begin{figure}
{\tiny
\begin{displaymath}
\xymatrixrowsep{1mm}
\xymatrixcolsep{1.5mm}
\xymatrix{
&&\alpha&&&\qquad&&&\Psi(\alpha)&&\\
\ar@{--}[dddddddddd]\ar@{--}[rrrr]&&&&\ar@{--}[dddddddddd]&
&\ar@{--}[dddddddddd]\ar@{--}[rrrr]&&&&\ar@{--}[dddddddddd]\\
\bullet\ar@{-}[rrrr]&&&&\bullet&&\bullet\ar@(r,ur)[]\ar@{<->}[rrrr]
&&&&\bullet\ar@(l,ul)[]\\
&&&&&&&&&&\\
&&&&\bullet&&&&&&
\bullet\ar@(l,ul)[]\\
&&&&&&&&&&\\
\bullet&&&&\bullet\ar@{-}@/^1pc/[uu]&&
\bullet\ar@(r,ur)[]&&&&
\bullet\ar@(l,ul)[]\ar@/^1.5pc/@{<->}[uu]\\
&&&&&&&&&&\\
&&&&\bullet&&&&&&\bullet\ar@(l,ul)[]\\
&&&&&&&&&&\\
\bullet\ar@{-}[rrrruuuu]&&&&\bullet\ar@{-}@/^1pc/[uu]&&
\bullet\ar@(r,dr)[]\ar@/_/@{<->}[rrrruuuu]\ar@{<->}[rrrruuuuuu]
&&&&\bullet\ar@(l,dl)[]\ar@/^1pc/@{<->}[uu]\\
\ar@{--}[rrrr]&&&&&&\ar@{--}[rrrr]&&&&
}
\end{displaymath}
\caption{Inclusion $\Psi$}
\label{fig8}
}
\end{figure}
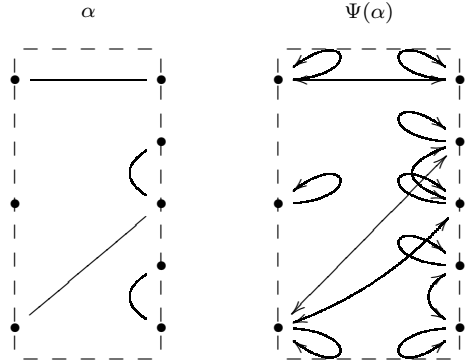

Partition category contains many classical subcategories, for
example, Brauer category (\cite{Br}), partial (alias rook)
Brauer category (\cite{Maz1}) and Temperley-Lieb category 
(\cite{TL}). The map $\Psi$ embeds them into $\mathfrak{PB}$ by 
restriction.

\section{Deformation}\label{s4}

In this section we establish existence of a $1$-parameter
deformation of the category $\mathfrak{PB}$.

\subsection{Frothy elements, edges and alternating cycles}\label{s4.1}

Let $\aleph=(\alpha_1,\alpha_2,\alpha_3,\dots,\alpha_k)$ be a {\em composable}
sequence of PBRs (see Subsection~\ref{s2.2}). Let $X_i$, $i=1,2,\dots,k+1$ 
be as in Subsection~\ref{s2.2}.  An edge $(a,b)\in\alpha_i$, 
$i=1,2,\dots,k$, is said to be $\aleph$-{\em frothy} provided that it does not
occur in any $\aleph$-connected sequence connecting two
(not necessarily distinct) elements from $X_1\coprod X_{k+1}$.
For example, in the case $k=2$ shown in Figure~\ref{fig9} all frothy
edges are drawn doubled.

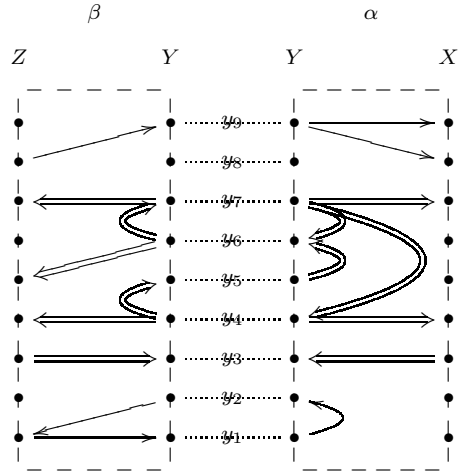
\begin{figure}
{\tiny
\begin{displaymath}
\xymatrixrowsep{1mm}
\xymatrixcolsep{1.5mm}
\xymatrix{
&&\beta&&&\qquad&&&\alpha&&\\
Z&&&&Y&&Y&&&&X\\
\ar@{--}[dddddddddd]\ar@{--}[rrrr]&&&&\ar@{--}[dddddddddd]&
&\ar@{--}[dddddddddd]\ar@{--}[rrrr]&&&&\ar@{--}[dddddddddd]\\
\bullet&&&&\bullet\ar@{.}[rr]&\text{\tiny$y_9$}&
\bullet\ar[rrrr]\ar[rrrrd]&&&&\bullet\\
\bullet\ar[rrrru]&&&&\bullet\ar@{.}[rr]&\text{\tiny${y}_8$}&
\bullet&&&&\bullet\\
\bullet&&&&\bullet\ar@{.}[rr]\ar@{=>}[llll]&\text{\tiny${y}_7$}&
\bullet\ar@{=>}[rrrr]\ar@{=>}@/^4pc/[ddd]\ar@{=>}@/^1.5pc/[d]&&&&\bullet\\
\bullet&&&&\bullet\ar@{.}[rr]\ar@{=>}[lllld]\ar@{=>}@/^1.5pc/[u]
&\text{\tiny${y}_6$}&\bullet&&&&\bullet\\
\bullet&&&&\bullet\ar@{.}[rr]&\text{\tiny${y}_5$}&
\bullet\ar@{=>}@/_1.5pc/[u]&&&&\bullet\\
\bullet&&&&\bullet\ar@{.}[rr]\ar@{=>}[llll]\ar@{=>}@/^1.5pc/[u]
&\text{\tiny${y}_4$}&\bullet\ar@{=>}[rrrr]&&&&\bullet\\
\bullet\ar@{=>}[rrrr]&&&&\bullet\ar@{.}[rr]&\text{\tiny${y}_3$}&
\bullet&&&&\bullet\ar@{=>}[llll]\\
\bullet&&&&\bullet\ar@{.}[rr]\ar[lllld]&\text{\tiny$y_2$}&\bullet&&&&\bullet\\
\bullet\ar[rrrr]&&&&\bullet\ar@{.}[rr]&\text{\tiny$y_1$}&
\bullet\ar@/_1.5pc/[u]&&&&\bullet\\
\ar@{--}[rrrr]&&&&&&\ar@{--}[rrrr]&&&&
}
\end{displaymath}
\caption{Example with $(\alpha,\beta)$-frothy edges drawn doubled}
\label{fig9}
}
\end{figure}

An $\aleph$-connected sequence $(a_1,b_1)$, $(a_2,b_2)$,\dots,  
$(a_m,b_m)$, where $m\in\{2,3,4,\dots\}$, is called an 
{\em  $\aleph$-frothy cycle} provided that the following conditions 
are satisfied:
\begin{enumerate}[(I)]
\setcounter{enumi}{2}
\item\label{c3} $a_1=b_m$ as elements of $X_{\coprod}$;
\item\label{c4} $(a_1,b_1)$ and $(a_m,b_m)$ come from different PBRs;
\item\label{c5} all edges $(a_i,b_i)$, $i=1,2,\dots,m$, are 
$\aleph$-frothy. 
\end{enumerate}
Directly from the definition we have that a cyclic permutation of
an $\aleph$-frothy cycle is again an $\aleph$-frothy cycle
(note here importance of condition \eqref{c4} to guarantee preservation 
of condition \eqref{c1}). We will call two $\aleph$-frothy cycles
{\em na{\"\i}vely equivalent} if they can be obtained from each other by 
a cyclic permutation. In what follows we will call a na{\"\i}ve equivalence 
class of $\aleph$-frothy cycles simply a {\em frothy cycle}
(if $\aleph$ is clear from the context). 

Two frothy cycles are called {\em elementary-equivalent}
provided that they contain a common edge and in both cycles this
edge appears as an edge of the same PBR (note that the relation of
elementary-equivalence is both symmetric and reflexive but not 
transitive in general). 
In the example shown in Figure~\ref{fig9} 
the two frothy cycles  $(y_6,y_7),(y_7,y_6)$ and 
$(y_4,y_5),(y_5,y_6),(y_6,y_7),(y_7,y_4)$
are elementary-equivalent. Finally, two frothy cycles 
$\xi$ and $\zeta$ are  called {\em equivalent} provided that there is
a sequence $\xi=\xi_0,\xi_1,\dots,\xi_k=\zeta$ of frothy cycles
for some $k\in\mathbb{N}$ such that every pair of consecutive 
frothy cycles in this sequence is elementary-equivalent. This 
is the minimum equivalence relations containing the relation of 
elementary-equivalence.

Write $M_{\aleph}$ for the set of equivalence classes of 
$\aleph$-frothy cycles; and define  $\mathfrak{f}(\aleph) = |M_{\aleph}|$. 
By definition, every
frothy edge appears in at most one equivalence class of 
frothy cycles, which implies that $\mathfrak{f}(\aleph)$ is finite.
In the example shown in Figure~\ref{fig9} we have
$\mathfrak{f}((\alpha,\beta))=1$.

\begin{proposition}\label{prop3}
Let $\alpha\in\mathfrak{PB}(X,Y)$, $\beta\in \mathfrak{PB}(Y,Z)$
and $\gamma\in \mathfrak{PB}(Z,U)$. Then
\begin{equation}\label{eq1}
\mathfrak{f}((\beta\circ\alpha,\gamma))+ \mathfrak{f}((\alpha,\beta))=
\mathfrak{f}((\alpha,\beta,\gamma))=
\mathfrak{f}((\alpha,\gamma\circ\beta))+ \mathfrak{f}((\beta,\gamma)).
\end{equation}
\end{proposition}

\begin{proof}
We prove the left equality. The right equality then follows applying
the involution $\star$. Set $\aleph:=(\alpha,\beta,\gamma)$.
Then let $M_{\alpha\beta} \subset M_{\aleph}$ be the subset of 
equivalence classes of $\aleph$-frothy 
cycles satisfying the condition that every frothy cycle in the
class contains only edges from $\alpha$ and $\beta$.
Define $M_{\gamma}$ as the complement, so that $M_{\aleph} =
M_{\alpha\beta} \sqcup M_{\gamma}$. It is easy to see that 
$M_{\alpha\beta}$ can be alternatively described as  
the set of equivalence classes containing an $\aleph$-frothy cycle 
all edges of which are $(\alpha,\beta)$-frothy and hence
$|M_{\alpha\beta}| = \mathfrak{f}((\alpha,\beta))$. 

It remains to show that $|M_{\gamma}| = 
\mathfrak{f}((\beta\circ\alpha,\gamma)):=|M_{(\beta\circ\alpha,\gamma)}|$. 
For this it is enough to establish a bijection 
$F: M_{(\beta\circ\alpha,\gamma)} \rightarrow M_{\gamma}$.
Note that an $\aleph$-frothy cycle  belonging to a class in $M_{\gamma}$
may contain no edges from $\gamma$. However, in this case it contains
at least one edge from $\alpha$ or $\beta$, which is not
$(\alpha,\beta)$-frothy (since there must be another cycle in 
its class that passes via $\gamma$).

We now construct $F$. 
Given a $(\beta\circ\alpha,\gamma)$-frothy cycle $\omega$, we 
substitute every $\beta\circ\alpha$-edge $(a,b)$ in $\omega$ by an
$(\alpha,\beta)$-connected sequence connecting $a$ to $b$. The 
obtained sequence $(a_1,a_2),\dots,(a_k,b_k)$ obviously 
satisfies \eqref{c1}--\eqref{c4}. We claim that  it also satisfies 
\eqref{c5}, that is that all $(a_i,b_i)$ are $\aleph$-frothy. Since
equivalence classes contain na{\"\i}ve equivalence classes, it is enough 
to show that $(a_1,b_1)$ is $\aleph$-frothy. Assume not, and let
$\omega_1,(a_1,b_1),\omega_2$ be an $\aleph$-connected sequence
connecting two elements of $X\coprod U$ (here $\omega_1$ and $\omega_2$
are two $\aleph$-connected sequences). Then the sequence 
\begin{displaymath}
\xi:=\omega_1,(a_1,b_1),(a_2,b_2),\dots,(a_k,b_k),(a_1,b_1),\omega_2 
\end{displaymath}
is again 
$\aleph$-connected connecting the same two elements of $X\coprod U$.
By definition, the original $(\beta\circ\alpha,\gamma)$-frothy cycle $\omega$
contained at least one edge from $\gamma$, say $(s,t)$. By construction, 
this edge appears in $\xi$. Applying to $\xi$ the procedure described in 
the proof of  Proposition~\ref{thm1}, we obtain a
$(\beta\circ\alpha,\gamma)$-connected sequence which connects two elements
from $X\coprod U$ and contains $(s,t)$. This means that $(s,t)$ is
not $(\beta\circ\alpha,\gamma)$-frothy, a contradiction. As the result,
$(a_1,a_2),\dots,(a_k,b_k)$ is an $\aleph$-frothy cycle. It is of the second
type as it contains an edge from $\gamma$. Clearly,
equivalent $(\beta\circ\alpha,\gamma)$-frothy cycles are mapped to
equivalent $\aleph$-frothy cycles and hence we obtain a map from
$M_{(\beta\circ\alpha,\gamma)}$ to $M_{\gamma}$.  

Now given an equivalence class in  $M_{\gamma}$,
choose a representative $\omega$, containing some edge from $\gamma$.
Using the na{\"\i}ve equivalence, we may assume that the first edge in 
$\omega$ is from $\gamma$. Substitute in $\omega$ every maximal subsequence 
of consecutive edges  from $\alpha$  and $\beta$ by the pair of elements 
which this sequence connects. The result will be an
$(\beta\circ\alpha,\gamma)$-connected cycle and, using the arguments
as in the previous paragraph, one shows that this cycle is frothy. 
For this procedure to define
a map from  $M_{\gamma}$ to $M_{(\beta\circ\alpha,\gamma)}$ we thus are left to check that equivalent
$\aleph$-frothy cycles are mapped to equivalent 
$(\beta\circ\alpha,\gamma)$-frothy cycles. By construction, two
elementary-equivalent $\aleph$-frothy cycles sharing an edge from $\gamma$
are mapped to elementary-equivalent $(\beta\circ\alpha,\gamma)$-frothy cycles.
To proceed we will use the following lemma:

\begin{lemma}\label{lem6}
Let $\omega'$ and $\omega''$ be equivalent $\aleph$-frothy cycles.
Then there exists an  $\aleph$-frothy cycle $\omega$
containing all the edges of both.
\end{lemma}

\begin{proof}
For 
$\aleph$-frothy cycles $\xi$ and $\xi'$
sharing some edge $(s,t)$ we may write $\xi=\xi_1,(s,t),\xi_2$ and
$\xi'=\xi'_1,(s,t),\xi'_2$. 
Then denote by $\xi\boxdot\xi'$ the 
$\aleph$-frothy cycle $\xi_1,(s,t),\xi'_2,\xi'_1,(s,t),\xi_2$.
Now let $\omega'=\omega_1,\omega_2,\dots,\omega_m=\omega''$ be a sequence
of $\aleph$-frothy cycles such that every pair of consecutive cycles
is elementary-equivalent, with a given shared edge;
and take 
\begin{displaymath}
\omega:= (\dots((\omega_1\boxdot\omega_2)\boxdot
\omega_3)\boxdot\dots)\boxdot\omega_m. 
\end{displaymath}
\end{proof}

Let $\omega'$ and $\omega''$ be equivalent 
$\aleph$-frothy cycles, each containing some edge from $\gamma$,
and $\omega$ be an $\aleph$-frothy cycle given by Lemma~\ref{lem6}.
Then $\omega'$ and $\omega$ are elementary-equivalent, as are 
$\omega''$ and $\omega$. By the paragraph
preceding Lemma~\ref{lem6}, we have that $\omega'$ and $\omega$
are mapped to elementary-equivalent $(\beta\circ\alpha,\gamma)$-frothy 
cycles, as are $\omega''$ and $\omega$. It follows that the
images of $\omega'$ and $\omega''$ are equivalent, giving us
a well-defined map from $M_{\gamma}$ to $M_{(\beta\circ\alpha,\gamma)}$.

From their constructions it follows directly that the maps between 
$M_{\gamma}$ and $M_{(\beta\circ\alpha,\gamma)}$ are mutually inverse 
bijections. This completes the proof.
\end{proof}

\subsection{Deformed category}\label{s4.2}

We consider $\mathbb{N}_0$ as an additive monoid in the natural way.
Consider the category $\overline{\mathfrak{PB}}$ defined as
follows: objects of $\overline{\mathfrak{PB}}$ are the same as
objects of $\mathfrak{PB}$; for $X,Y\in \overline{\mathfrak{PB}}$
the morphism set $\overline{\mathfrak{PB}}(X,Y)$ equals
$\mathfrak{PB}(X,Y)\times \mathbb{N}_0$; for 
$(\alpha,k)\in \overline{\mathfrak{PB}}(X,Y)$ and
$(\beta,m)\in \overline{\mathfrak{PB}}(Y,Z)$ set
\begin{equation}\label{eq55}
(\beta,m)\diamond(\alpha,k):=
(\beta\circ\alpha,m+k+\mathfrak{f}(\alpha,\beta)).
\end{equation}

\begin{theorem}\label{thm5}
The above definition makes  $\overline{\mathfrak{PB}}$ into
a category.
\end{theorem}

\begin{proof}
Associativity of $\diamond$ follows from Proposition~\ref{prop3}.
Note that the identity morphism 
$\varepsilon_X$ in $\mathfrak{PB}(X,X)$ does not
have any edges connecting two elements of the codomain.
This implies that for any $\alpha\in \mathfrak{PB}(X,Y)$ we have
$\mathfrak{f}(\varepsilon_X,\alpha)=\mathfrak{f}(\alpha,\varepsilon_Y)=0$.
Hence  $(\varepsilon_X,0)$ is the identity morphism  in 
$\overline{\mathfrak{PB}}(X,X)$. The claim follows.
\end{proof}

\subsection{Deformed partition category via restriction}\label{s4.3}

Recall, from \cite{Mar1}, that the category $\mathfrak{P}$ admits 
deformation $\overline{\mathfrak{P}}$, similar to the deformation 
$\overline{\mathfrak{PB}}$ of $\mathfrak{PB}$. It is constructed as
follows: The category $\overline{\mathfrak{P}}$ has the same objects
as $\mathfrak{P}$. For $X,Y\in \overline{\mathfrak{P}}$ the set
$\overline{\mathfrak{P}}(X,Y)$ equals $\mathfrak{P}(X,Y)\times\mathbb{N}_0$
and the multiplication in $\overline{\mathfrak{P}}(X,Y)$ is given for
$(\alpha,k)\in \overline{\mathfrak{P}}(X,Y)$ and
$(\beta,m)\in \overline{\mathfrak{P}}(Y,Z)$ by the following:
\begin{displaymath}
(\beta,m)\diamond(\alpha,k):=
(\beta\circ\alpha,m+k+\mathfrak{p}(\alpha,\beta)),
\end{displaymath}
where $\mathfrak{p}(\alpha,\beta)$ is defined as follows: Denote by
$Y'$ the set of all $y\in Y$ for which there does {\em not} exist a
sequence $y=a_1,a_2,a_3,\dots,a_p$, where all $a_i\in X\coprod Y\coprod Z$,
$a_p\in X\coprod Z$, and such that every two 
consecutive elements in this sequence 
belong to the same part of either $\alpha$ or $\beta$. Introduce
an equivalence relation $\sim$ on $Y'$ as follows: 
$y_1\sim y_2$ for $y_1,y_2\in Y'$ if and only if there is  a
sequence $y_1=a_1,a_2,a_3,\dots,a_p=y_2$, where all $a_i\in Y$,
such that every two consecutive elements in this sequence 
belong to the same part of either $\alpha$ or $\beta$.
Then $\mathfrak{p}(\alpha,\beta)$ is defined as the number of equivalence
classes of $\sim$. Our main observation in this subsection is the
following statement which says that $\Psi$ can be lifted up to
the level of deformed categories.

\begin{proposition}\label{prop9}
Define $\overline{\Psi}:\overline{\mathfrak{P}}\to
\overline{\mathfrak{PB}}$ as the identity on objects and
$\overline{\Psi}((\alpha,k)):=(\Psi(\alpha),k)$ for any
morphism $(\alpha,k)$. Then 
\begin{displaymath}
\overline{\Psi}((\beta,m)\diamond(\alpha,k))=
\overline{\Psi}((\beta,m))\diamond\overline{\Psi}((\alpha,k)) 
\end{displaymath}
for all composable morphisms $(\alpha,k)$ and $(\beta,m)$ in
$\overline{\mathfrak{P}}$.
\end{proposition}

\begin{proof}
To prove this statement we need to check that for any
morphisms $\alpha\in\mathfrak{P}(X,Y)$ and $\beta\in\mathfrak{P}(Y,Z)$ 
there is a bijection between the set $M_1$ of equivalence classes for 
the relation  $\sim$ defined above and the set $M_2$ of equivalence 
classes of $(\Psi(\alpha),\Psi(\beta))$-frothy cycles. 
 
Every  $(\Psi(\alpha),\Psi(\beta))$-frothy cycle consists of edges 
between elements in $Y$. From the definition of $\Psi$ it follows easily that
all these elements, in fact, belong to $Y'$. Moreover, from the definition 
of $\sim$ it follows  that all these element are $\sim$-related. 
Hence we can define a map from the set of 
$(\Psi(\alpha),\Psi(\beta))$-frothy cycles
to $M_1$ by sending each cycle to the corresponding equivalence class
of $\sim$ described above. Since $\sim$ is an equivalence relation,
elementary equivalent cycles have the same image. This means that this
map factors through $M_2$ giving us a map from $M_2$ to $M_1$.

First of all we claim that this map is surjective. Indeed, given an
equivalence class $N$ of $\sim$, let $y\in N$.
Then the construction of $\Psi$ implies that the edge $(y,y)$
is contained both in $\Psi(\alpha)$ and $\Psi(\beta)$. Therefore
$(y,y),(y,y)$ is a $(\Psi(\alpha),\Psi(\beta))$-frothy cycle
(in which the first edge is in  $\Psi(\alpha)$
and the second edge is in $\Psi(\beta))$).
By construction, the cycle $(y,y),(y,y)$ is mapped to $N$, which
implies surjectivity.

Now we claim that our map is injective. Let $N$ be an equivalence
class of $\sim$. 
To prove the assertion we have to show that all
$(\Psi(\alpha),\Psi(\beta))$-frothy cycles mapped to $N$ are 
equivalent. For this it is enough to show that every
$(\Psi(\alpha),\Psi(\beta))$-frothy cycle mapped to $N$ is
equivalent to a $(\Psi(\alpha),\Psi(\beta))$-frothy cycle
of the form $(y,y),(y,y)$ as above; and that all  such
$(\Psi(\alpha),\Psi(\beta))$-frothy cycles are equivalent.

Let $\omega$ be am $(\Psi(\alpha),\Psi(\beta))$-frothy cycles
and $(s,t)$ its first edge. Then $(s,s),(s,s),\omega$, where the 
first edge $(s,s)$ is considered from  the same factor 
($\alpha$ or $\beta$) as the edge $(s,t)$ of $\omega$,
is a $(\Psi(\alpha),\Psi(\beta))$-frothy cycle, which is elementary
equivalent to $\omega$. On the other hand, the cycle
$(s,s),(s,s),\omega$ is elementary equivalent to $(s,s),(s,s)$.
Hence $\omega$ is equivalent to $(s,s),(s,s)$.

Now let $s,t\in N$ and $s=a_1,a_2,\dots,a_k=t$ be a sequence of 
elements from $Y'$ in which every pair of consecutive elements
belongs to the same part of either $\alpha$ or $\beta$. Without
loss of generality we may even assume that this alternates in
the sense that if $a_1$ and $a_2$ belong to the same edge of
$\alpha$, then $a_2$ and $a_3$ belong to the same edge of $\beta$
and so on. From the definition of $\Psi$ it follows that we have
a $(\Psi(\alpha),\Psi(\beta))$-connected sequence as follows:
$(a_1,a_2),(a_2,a_3),\dots,(a_{k-1},a_k)$. This yields existence of
a $(\Psi(\alpha),\Psi(\beta))$-frothy cycle as follows:
\begin{displaymath}
\omega:=(a_1,a_2),(a_2,a_3),\dots,(a_{k-1},a_k),(a_k,a_k),
(a_{k},a_{k-1}),\dots,(a_2,a_1),(a_1,a_1).
\end{displaymath}
Here $(a_i,a_{i-1})$ and $(a_{i-1},a_i)$ are considered as edges of
the same factor ($\alpha$ or $\beta$), $(a_k,a_k)$ is considered
as an edge from the factor, different from the factor containing 
$(a_{k-1},a_k)$, and $(a_1,a_1)$ is considered
as an edge from the factor, different from the factor containing 
$(a_{1},a_2)$. The cycle $\omega$ is elementary equivalent to both
$(s,s),(s,s)$ and $(t,t),(t,t)$, which implies that the latter two
cycles are equivalent. This yields injectivity.

The above implies that our map is bijective and the claim of the 
proposition follows.
\end{proof}

The deformation $\overline{\mathfrak{P}}$ 
of the partition category contains deformations
of both Brauer and Temperley-Lieb categories as well as the
one-parameter deformation of the partial Brauer category
(\cite{Maz2}). The map $\overline{\Psi}$ embeds them 
into $\overline{\mathfrak{PB}}$ by restriction. Some diagram
categories admit a two-parameter deformation, see
\cite{Maz2,MM,Mar3}. However, we do not know how to realize
these one in terms of the category $\overline{\mathfrak{PB}}$.

\subsection{Oriented Brauer and Temperley-Lieb 
categories}\label{s4.5}

For finite sets $X$ and $Y$ a PBR $\alpha\in\mathfrak{PB}(X,Y)$ is
called an {\em oriented partial Brauer diagram} provided that every 
element $s\in X\coprod Y$ appears in at most one edge of $\alpha$. 
An oriented partial Brauer diagram $\alpha$ for which every 
element $s\in X\coprod Y$ appears in exactly one edge of $\alpha$
is called an {\em oriented Brauer diagram}. An example of an oriented 
partial Brauer diagram is given in Figure~\ref{fig51}. 
One can say that an oriented (partial) Brauer diagram is obtained from a
usual (partial) Brauer diagram (see \cite{Br,Maz1}) by choosing orientation 
of all chords on the latter. A (partial) Brauer diagram is obtained from 
an oriented (partial) Brauer diagram by forgetting the orientation.

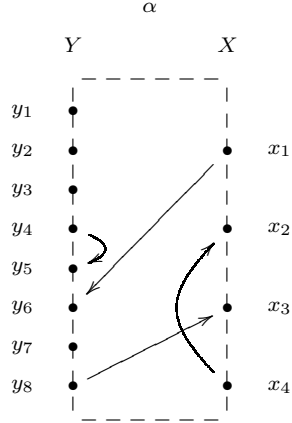
\begin{figure}
{\tiny
\begin{displaymath}
\xymatrixrowsep{1mm}
\xymatrixcolsep{1.5mm}
\xymatrix{
&&&\alpha&&&\\
&Y&&&&X&\\
&\ar@{--}[rrrr]\ar@{--}[ddddddddd]&&&&\ar@{--}[ddddddddd]&\\
y_1&\bullet&&&&&\\
y_2&\bullet&&&&\bullet\ar[lllldddd]&x_1\\
y_3&\bullet&&&&&\\
y_4&\bullet\ar@/^1pc/[d]&&&&\bullet&x_2\\
y_5&\bullet&&&&&\\
y_6&\bullet&&&&\bullet&x_3\\
y_7&\bullet&&&&&\\
y_8&\bullet\ar[rrrruu]&&&&\bullet\ar@(ul,dl)[uuuu]&x_4\\
&\ar@{--}[rrrr]&&&&&
}
\end{displaymath}
\caption{Oriented partial Brauer diagram}
\label{fig51}
}
\end{figure}

\begin{lemma}\label{lem52}
Let $\alpha$ be a PBR on $(X,Y)$ and $\beta$ be a PBR on $(Y,Z)$. 
Assume that both $\alpha$ and $\beta$ are oriented partial
Brauer diagrams. Then we have the following:
\begin{enumerate}[$($a$)$]
\item\label{lem52.1} The composition $\beta\circ\alpha$ is
an oriented partial Brauer diagram. 
\item\label{lem52.2} The number $\mathfrak{f}((\alpha,\beta))$ 
is the number of oriented cycles on the diagram from  Figure~\ref{fig9}.  
\end{enumerate}
\end{lemma}

\begin{proof}
Any element of $X$ and $Z$ appears in at most one edge of $\alpha$
or $\beta$, respectively. 
Any element of $Y$ appears in at most one edge of $\alpha$ and
in at most one edge of $\beta$. Hence for every $s\in X\coprod Z$,
there is at most one $(\alpha,\beta)$-connected sequence
connecting $s$ to some element $t\in X\coprod Z$, moreover, $s\neq t$.
This implies both claim \eqref{lem52.1} and the fact that every
equivalence class of $(\alpha,\beta)$-frothy cycles consists of
a single element. The latter implies claim \eqref{lem52.2}.
\end{proof}

The collection of all oriented partial Brauer diagrams does not give rise to
a subcategory of $\mathfrak{PB}$ (or $\overline{\mathfrak{PB}}$) because of
the absence of identity morphisms. The collection of all oriented  
Brauer diagrams is not even closed under composition (the composition of
two oriented Brauer diagrams is only an oriented partial Brauer diagram
in general). One can remedy the situation in the following way
(confer \cite{RT}).

Define the category $\mathfrak{O}$ as follows:
Objects of $\mathfrak{O}$ are pairs $\mathbf{X}:=(X_1,X_2)$
of finite sets such that $X_1\subset X_2$. For 
$\mathbf{X},\mathbf{Y}\in \mathfrak{O}$ the set 
$\mathfrak{O}(\mathbf{X},\mathbf{Y})$ consists of all pairs
$(\alpha,k)$, where $k\in \mathbb{N}_0$ and $\alpha$ is an oriented
Brauer diagram $\alpha$ on $(X_2,Y_2)$ such that the following 
condition is satisfied:
\begin{equation}\label{eq56}
\text{For every edge }  (a,b)\in \alpha \text{ we have } 
a\in X_1\cup(Y_2\setminus Y_1)\text{ and } b\in Y_1\cup (X_2\setminus X_1). 
\end{equation}
For $(\alpha,k)\in \mathfrak{O}(\mathbf{X},\mathbf{Y})$ and
$(\beta,m)\in \mathfrak{O}(\mathbf{Y},\mathbf{Z})$ define the composition
$(\beta,m)\diamond  (\alpha,k)$ by formula \eqref{eq55}.
For $\mathbf{X}\in \mathfrak{O}$ denote by 
$\check{\varepsilon}_{\mathbf{X}}$ the oriented Brauer 
diagram of the identity morphism for $\mathbf{X}$. 
This diagram consists of all edges
$(x^{(d)},x^{(c)})$, $x\in X_1$,  and $(x^{(c)},x^{(d)})$, 
$x\in X_2\setminus X_1$, see example in Figure~\ref{fig57}.

\begin{figure}
{\tiny
\begin{displaymath}
\xymatrixrowsep{3mm}
\xymatrixcolsep{1.5mm}
\xymatrix{
&\ar@{--}[dddddd]\ar@{--}[rrrr]&&&&\ar@{--}[dddddd]&\\
x_1&\bullet&&&&\bullet\ar[llll]&x_1\\
x_2&\bullet\ar[rrrr]&&&&\bullet&x_2\\
x_3&\bullet\ar[rrrr]&&&&\bullet&x_3\\
x_4&\bullet&&&&\bullet\ar[llll]&x_4\\
x_5&\bullet&&&&\bullet\ar[llll]&x_5\\
&\ar@{--}[rrrr]&&&&
}
\end{displaymath}
\caption{The oriented Brauer diagram 
$\check{\varepsilon}_{(\{x_1,x_4,x_5\},\{x_1,x_2,x_3,x_4,x_5\})}$}
\label{fig57}
}
\end{figure}
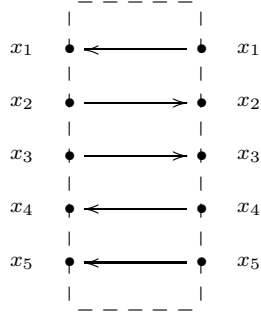

\begin{proposition}\label{prop56}
The construct $\mathfrak{O}$ above is a category, 
called {\em oriented Brauer category}.
\end{proposition}

\begin{proof}
For $(\alpha,k)\in \mathfrak{O}(\mathbf{X},\mathbf{Y})$ and
$(\beta,m)\in \mathfrak{O}(\mathbf{Y},\mathbf{Z})$, from the definition 
of $\mathfrak{O}$ it follows immediately that $\beta\circ\alpha$
is an oriented Brauer diagram. Now associativity is obtained 
from Theorem~\ref{thm5} by restriction. The fact that the 
$\check{\varepsilon}_{\mathbf{X}}$'s are identity morphisms is proved by 
a straightforward computation.
\end{proof}

The (standard skeleton of) classical Brauer category has a natural 
topological counterpart, known as the category of tangles 
(see e.g. \cite{Tu2}). The natural topological counterpart
of the category $\mathfrak{O}$ is the category of oriented tangles, 
see \cite{Tu}. The corresponding planar objects are the Temperley-Lieb
and the oriented Temperley-Lieb categories. To define the 
oriented Temperley-Lieb category $\mathfrak{OTL}$ for every finite set
$X$ fix a linear order $<_X$ on $X$. Then the category $\mathfrak{OTL}$
is defined as the subcategory of $\mathfrak{O}$ with the same set
of objects and containing all those morphisms $(\alpha,k)$ for which
the diagram of $\alpha$ can be drawn planar 
(whenever the elements of the domain 
and the codomain are listed with respect to the fixed linear order from
top to bottom). Similarly one defines the partial oriented Brauer category 
$\mathcal{P}\mathfrak{O}$ and the partial oriented Temperley-Lieb category 
$\mathcal{P}\mathfrak{OTL}$.

\section{Polarized factorization}\label{s5}

In this section we establish a factorization  of
partitioned  binary relations, called {\em polarized factorization}.

\subsection{Pure partitioned binary relations}\label{s5.1}

Let $X,Y\in\mathfrak{PB}$ and $\alpha\in\mathfrak{PB}(X,Y)$. The
PBR $\alpha$ is called {\em pure} provided that every edge in 
$\alpha$ consists of an element in $\mathrm{Dom}(\alpha)$ and
an element in $\mathrm{Codom}(\alpha)$. For example, both PBRs 
$\varepsilon_X$ and $\hat{\varepsilon}_X$ are pure while the
PBR $\overline{\varepsilon}_X$ is not pure (see Figure~\ref{fig3}). 
Another example of a 
pure PBR is shown in Figure~\ref{fig71} in the middle.

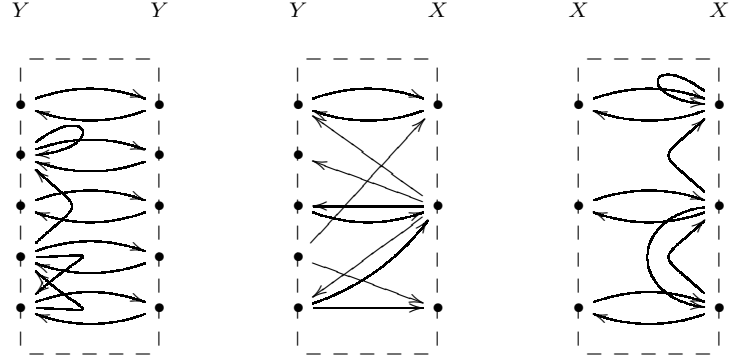
\begin{figure}
{\tiny
\begin{displaymath}
\xymatrixrowsep{3mm}
\xymatrixcolsep{1.5mm}
\xymatrix{
Y&&&&Y&&&&Y&&&&X&&&&X&&&&X\\
\ar@{--}[dddddd]\ar@{--}[rrrr]&&&&\ar@{--}[dddddd]
&&&&\ar@{--}[rrrr]\ar@{--}[dddddd]&&&&\ar@{--}[dddddd]
&&&&\ar@{--}[rrrr]\ar@{--}[dddddd]&&&&\ar@{--}[dddddd]\\
\bullet\ar@/^/[rrrr]&&&&\bullet\ar@/^/[llll]&&&&
\bullet\ar@/^/[rrrr]&&&&\bullet\ar@/^/[llll]&&&&
\bullet\ar@/^/[rrrr]&&&&\bullet\ar@/^/[llll]\ar@(ul,l)[]\\
\bullet\ar@/^/[rrrr]\ar@(ur,r)[]&&&&\bullet\ar@/^/[llll]&&&&
\bullet&&&&&&&&&&&&\\
\bullet\ar@/^/[rrrr]&&&&\bullet\ar@/^/[llll]&&&&
\bullet\ar@/_/[rrrr]&&&&\bullet\ar[llll]\ar[llllu]
\ar[lllluu]\ar[lllldd]&&&&
\bullet\ar@/^/[rrrr]&&&&\bullet\ar@/^/[llll]\ar@(l,l)[dd]\ar@(ul,dl)[uu]\\
\bullet\ar@/^/[rrrr]\ar@(ur,dr)[uu]\ar@(r,ur)[d]&&&&
\bullet\ar@/^/[llll]&&&&\bullet\ar[rrrrd]\ar[rrrruuu]&&&&&&&&&&&&\\
\bullet\ar@/^/[rrrr]\ar@(r,dr)[u]&&&&\bullet\ar@/^/[llll]&&&&
\bullet\ar[rrrr]\ar@/_/[rrrruu]&&&&\bullet&&&&
\bullet\ar@/^/[rrrr]&&&&\ar@(ul,dl)[uu]\bullet\ar@/^/[llll]\\
\ar@{--}[rrrr]&&&&&&&&\ar@{--}[rrrr]&&&&&&&&\ar@{--}[rrrr]&&&&
}
\end{displaymath}
\caption{Left polarized idempotent, pure PBR and 
right polarized idempotent}
\label{fig71}
}
\end{figure}

\begin{lemma}\label{lem72}
The composition of two composable pure PBRs is pure. Hence, taking 
all pure PBRs as morphisms defines a subcategory of $\mathfrak{PB}$
of pure PBRs, which we will denote by $\mathfrak{PPB}$. 
\end{lemma}

\begin{proof}
As the PBR $\varepsilon_X$ of the identity morphism
is pure, to prove the claim we
have only to check that pure PBRs are closed with respect to 
composition. This follows directly from definitions.
\end{proof}

The category $\overline{\mathfrak{PB}}$ of Section~\ref{s4.2}
contains a subcategory
$\mathfrak{P}\overline{\mathfrak{PB}}$ which has the same objects 
as $\overline{\mathfrak{PB}}$ and whose morphisms are all morphisms
of the form $(\alpha,0)$, where $\alpha$ is a morphism from 
$\mathfrak{PPB}$. It is easy to see that no frothy cycles appear 
when composing two pure PBRs, and hence the categories 
$\mathfrak{P}\overline{\mathfrak{PB}}$ and $\mathfrak{PPB}$ are isomorphic.

The category $\mathfrak{PPB}$ admits a nice description in terms of
the category $\mathfrak{B}$ of binary relations.  
Consider the {\em double} $\mathfrak{B}^{\ltimes}$ 
of the category $\mathfrak{B}$ defined as follows: Objects of
$\mathfrak{B}^{\ltimes}$ are the same as objects of $\mathfrak{B}$. For
$X,Y\in \mathfrak{B}^{\ltimes}$ the set $\mathfrak{B}^{\ltimes}(X,Y)$ 
consists of
pairs $(\beta,\gamma)$, where $\beta\in \mathfrak{B}(X,Y)$ and
$\gamma\in \mathfrak{B}^{\mathrm{op}}(X,Y)$ (the opposite category).
For $(\beta,\gamma)\in \mathfrak{B}^{\ltimes}(X,Y)$ and
$(\beta',\gamma')\in \mathfrak{B}^{\ltimes}(Y,Z)$ the composition
is defined as follows:
\begin{displaymath}
(\beta',\gamma')(\beta,\gamma)=
(\beta'\beta,\gamma\gamma').
\end{displaymath}

\begin{proposition}\label{prop73}
The categories $\mathfrak{B}^{\ltimes}$ and $\mathfrak{PPB}$ are isomorphic.
\end{proposition}

\begin{proof}
By definition, these categories have the same objects. For
$\alpha\in \mathfrak{PPB}(X,Y)$, where $X,Y\in \mathfrak{PPB}$, let
$\beta\in \mathfrak{B}(X,Y)$ be the collection of all edges
$(a,b)$ of $\alpha$ such that $a\in\mathrm{Dom}(\alpha)$ and
$b\in\mathrm{Codom}(\alpha)$. Let $\gamma\in \mathfrak{B}^{\mathrm{op}}(X,Y)$ 
be the collection of all edges $(a,b)$ of $\alpha$ such that 
$a\in\mathrm{Codom}(\alpha)$ and $b\in\mathrm{Dom}(\alpha)$.
From the definition of pure PBRs it follows easily that the
map $\alpha\mapsto(\beta,\gamma)$ is a bijection from
$\mathfrak{PPB}(X,Y)$ to $\mathfrak{B}^{\ltimes}(X,Y)$. It is also easy to 
check that this map is compatible with compositions on both 
sides. The claim follows.
\end{proof}

\begin{remark}\label{rem73n}
{\rm
Under the identification of $\mathfrak{B}^{\ltimes}$
and $\mathfrak{PPB}$ from Proposition~\ref{prop73},
the ``diagonal'' image of $\mathfrak{B}$ in $\mathfrak{B}^{\ltimes}$
given by $\alpha\mapsto(\alpha,\alpha^{\bowtie})$ coincides with 
$\Phi_2(\mathfrak{B})$ (see Subsection~\ref{s3.2}).
This implies Proposition~\ref{prop4}.
}
\end{remark}

\subsection{Left and right polarized idempotents}\label{s5.2}

Let $X\in\mathfrak{PB}$ and $\alpha\in\mathfrak{PB}(X,X)$. The element
$\alpha$ is called a {\em left polarized idempotent} provided
that $\alpha$ contains all edges from $\varepsilon_X$ and any other
edge of $\alpha$ has the form $(a,b)$, where $a,b\in\mathrm{Codom}(\alpha)$.
Define a {\em right polarized idempotent} similarly using 
$\mathrm{Dom}(\alpha)$. It is easy to see that every left (right)
polarized idempotent is indeed an idempotent. In particular, the 
identity morphism $\varepsilon_X$ is both, left and right, polarized.
An example of a left polarized idempotent is
given in Figure~\ref{fig71} on the left. An example of a right polarized 
idempotent is given in Figure~\ref{fig71} on the right. We denote by
$PI(X,l)$ and $PI(X,r)$ the sets of left and right polarized idempotents 
in $\mathfrak{PB}(X,X)$, respectively.

\begin{lemma}\label{lem75}
Both $PI(X,l)$ and  $PI(X,r)$ are  submonoids of $\mathfrak{PB}(X,X)$
isomorphic to the commutative band (semilattice) $(\mathfrak{B}(X,X),\cup)$.
In particular, we have $|PI(X,l)|=|PI(X,r)|=2^{|X|^2}$.
\end{lemma}

\begin{proof}
Straightforward computation.
\end{proof}

\subsection{Polarized factorization of partitioned 
binary relations}\label{s5.3}

Let $X,Y\in\mathfrak{PB}$ and $\alpha\in\mathfrak{PB}(X,Y)$.
Define $\gamma_{\alpha}$ as the pure PBR on $(X,Y)$
consisting of all edges from $\alpha$, which contain an element 
in $\mathrm{Dom}(\alpha)$ and an element in $\mathrm{Codom}(\alpha)$.
Define $\beta_{\alpha}$ as the left polarized idempotent in 
$\mathfrak{PB}(Y,Y)$ such that for every edge $(a,b)$, where 
$a,b\in Y=\mathrm{Codom}(\alpha)$ we have $(a,b)\in \alpha$ if and
only if $(a,b)\in\beta_{\alpha}$. Define $\delta_{\alpha}$ as the 
right  polarized idempotent in 
$\mathfrak{PB}(X,X)$ such that for every edge $(a,b)$, where 
$a,b\in X=\mathrm{Dom}(\alpha)$ we have $(a,b)\in \alpha$ if and
only if $(a,b)\in\delta_{\alpha}$.

The main result of this section is the following statement 
establishing {\em polarized factorization} of 
partitioned  binary relations.

\begin{theorem}\label{thm77}
Let $X,Y\in\mathfrak{PB}$, $\alpha\in\mathfrak{PB}(X,Y)$
and $\beta_{\alpha}$, $\gamma_{\alpha}$ and $\delta_{\alpha}$ be as 
defined above. Then 
$\alpha=\beta_{\alpha}\circ\gamma_{\alpha}\circ\delta_{\alpha}$ 
is the unique factorization  of $\alpha$ into a product of a left 
polarized idempotent, a pure PBR and a right polarized idempotent. 
\end{theorem}

\begin{proof}
That  $\alpha=\beta_{\alpha}\circ\gamma_{\alpha}\circ\delta_{\alpha}$ 
is checked by a  straightforward computation, proving existence.
Having established existence, uniqueness is proved by a counting
argument. Indeed, we have $|\mathfrak{PB}(X,Y)|=2^{(|X|+|Y|)^2}$.
At the same time, the number of left polarized idempotents in 
$\mathfrak{PB}(Y,Y)$ equals $2^{|Y|^2}$, the number of right 
polarized idempotents in  $\mathfrak{PB}(X,X)$ equals $2^{|X|^2}$
and the number of pure PBRs on $(X,Y)$ equals $2^{2|X||Y|}$.
Hence the multiplication rule implies that
\begin{displaymath}
|\mathfrak{PB}(X,Y)|=|PI(Y,l)\times \mathfrak{PPB}(X,Y)\times PI(X,r)|
\end{displaymath}
and the claim follows.
\end{proof}

It is easy to see that the polarized factorization in 
$\mathfrak{PB}$ gives rise to a factorization in $\overline{\mathfrak{PB}}$.
Theorem~\ref{thm77} shows that morphisms of the relatively complicated 
category $\mathfrak{PB}$ decompose canonically into a product of morphisms
from the less complicated category $\mathfrak{PPB}$ and elements of 
some commutative bands. 

\subsection{Composition of PBRs via composition 
of binary relations}\label{s5.4}

The polarized decomposition of PBRs motivates the following construction:
For a PBR $\alpha$ consider the following subsets of $\alpha$:
\begin{displaymath}
\begin{array}{rcl}
\alpha_{11}&:=&\{(a,b)\in \alpha:a\in \mathrm{Dom}(\alpha), 
b\in \mathrm{Dom}(\alpha)\},\\  
\alpha_{12}&:=&\{(a,b)\in \alpha:a\in \mathrm{Dom}(\alpha), 
b\in \mathrm{Codom}(\alpha)\},\\  
\alpha_{21}&:=&\{(a,b)\in \alpha:a\in \mathrm{Codom}(\alpha), 
b\in \mathrm{Dom}(\alpha)\},\\  
\alpha_{22}&:=&\{(a,b)\in \alpha:a\in \mathrm{Codom}(\alpha), 
b\in \mathrm{Codom}(\alpha)\}. 
\end{array}
\end{displaymath}
Then $\alpha$ is a disjoint union of the $\alpha_{ij}$'s,
$i,j=1,2$. Moreover, the $\alpha_{ij}$'s can be interpreted in terms of 
factors of the polarized decomposition of $\alpha$ in the obvious way 
(i.e. $\gamma_{\alpha}=\alpha_{12}\cup\alpha_{21}$,
$\delta_{\alpha}=\varepsilon_X\cup \alpha_{11}$ and 
$\beta_{\alpha}=\varepsilon_Y\cup \alpha_{22}$).

Given a PBR $\beta$ composable with $\alpha$, directly from the 
definition of the product we obtain the following formulae:
\begin{equation}\label{eq46}
\begin{array}{rcc}
(\beta\circ\alpha)_{11}&=&\displaystyle \alpha_{11}\cup
\bigcup_{i\geq 1} \alpha_{21}\circ(\beta_{11}\circ\alpha_{22})^i
\circ\beta_{11}\circ\alpha_{12},\\
(\beta\circ\alpha)_{22}&=&\displaystyle \beta_{22}\cup
\bigcup_{i\geq 1} \beta_{12}\circ(\alpha_{22}\circ\beta_{11})^i
\circ\alpha_{22}\circ\beta_{21},\\
(\beta\circ\alpha)_{12}&=&\displaystyle  \bigcup_{i\geq 0} 
\beta_{12}\circ(\alpha_{22}\circ\beta_{11})^i \circ\alpha_{12},\\
(\beta\circ\alpha)_{21}&=&\displaystyle  \bigcup_{i\geq 0}
\alpha_{21}\circ(\beta_{11}\circ\alpha_{22})^i \circ\beta_{21}.
\end{array}
\end{equation}

\subsection{On random products of PBRs}\label{s5.5}

For a finite set $X$ denote by $\omega_X$ the maximum binary
relation on $X$ with respect to  inclusions (i.e. the {\em full}
relation). Denote also by $\overline{\omega}_X$ the maximum
PBR on $(X,X)$ with respect to inclusions. Let $A_X$ denote the set 
of all pairs $(\alpha,\alpha')\in \mathfrak{B}(X,X)\times \mathfrak{B}(X,X)$ 
such that $\alpha\circ\alpha'=\omega_X$. Let $\overline{A}_X$ denote the 
set  of all pairs $(\alpha,\alpha')\in 
\mathfrak{PB}(X,X)\times \mathfrak{PB}(X,X)$ 
such that $\alpha\circ\alpha'=\overline{\omega}_X$.
Recall the following classical result (see e.g. \cite[Theorem~4]{KR}):

\begin{proposition}\label{prop41}
We have:
\begin{displaymath}
\lim_{|X|\to\infty}\frac{|A_X|}{|\mathfrak{B}(X,X)\times 
\mathfrak{B}(X,X)|}=1. 
\end{displaymath}
\end{proposition}

%
%

Let $A'_X$ denote the set 
of all  $(\alpha,\alpha',\alpha'')\in 
\mathfrak{B}(X,X)\times \mathfrak{B}(X,X)\times \mathfrak{B}(X,X)$ 
such that $\alpha\circ\alpha'\circ\alpha''=\omega_X$. 

\begin{corollary}\label{cor44}
We have:
\begin{displaymath}
\lim_{|X|\to\infty}\frac{|A'_X|}{|\mathfrak{B}(X,X)\times 
\mathfrak{B}(X,X)\times  \mathfrak{B}(X,X)|}=1. 
\end{displaymath}
\end{corollary}

\begin{proof}
By Proposition~\ref{prop41}, when $|X|\to\infty$ both 
the probability of $\alpha\circ \alpha'=\omega_X$ and
of $\alpha'\circ\alpha''=\omega_X$ tend to $1$. Hence 
the probability of the intersection of these events  tends
to $1$ as well. However, if $\alpha'\circ\alpha''=\omega_X$,
then the Boolean matrix of $\alpha''$ cannot have zero columns.
Hence, in this case $\alpha\circ \alpha'=\omega_X$ implies
$\alpha\circ\alpha'\circ\alpha''=\omega_X$. The claim follows.
\end{proof}

\begin{remark}\label{rem43}
{\em
Proposition~\ref{prop41} combined with \cite[Theorem~6]{GM1} 
implies \cite[Conjecture~5]{GM2}.
}
\end{remark}

In the following statement we extend Proposition~\ref{prop41}
to PBRs.

\begin{theorem}\label{thm42}
We have:
\begin{displaymath}
\lim_{|X|\to\infty}\frac{|\overline{A}_X|}{|\mathfrak{PB}(X,X)
\times \mathfrak{PB}(X,X)|}=1. 
\end{displaymath}
\end{theorem}

\begin{proof}
By Subsection~\ref{s5.4}, choosing a PBR $\alpha$ is equivalent to
choosing four binary relations $\alpha_{ij}$, $i,j=1,2$. 
By \eqref{eq46}, $\beta\circ\alpha=\overline{\omega}_{X}$
is guaranteed by the following list of conditions:
\begin{displaymath}
\beta_{12}\circ\alpha_{12}=\omega_{X};\quad 
\alpha_{21}\circ\beta_{21}=\omega_{X};\quad 
\alpha_{21}\circ\beta_{11}\circ\alpha_{12}=\omega_{X};\quad 
\beta_{12}\circ\alpha_{22}\circ\beta_{21}=\omega_{X}.
\end{displaymath}
By Proposition~\ref{prop41} and Corollary~\ref{cor44}, when
$|X|\to\infty$, the probability of each of these conditions
tends to $1$. Hence the probability of their intersection tends
to $1$ as well. The claim follows.
\end{proof}

\vspace{0.3cm}

\noindent
P.~M.: Department of Pure Mathematics, University of Leeds, Leeds, 
LS2 9JT, UK, e-mail: {\tt ppmartin\symbol{64}maths.leeds.ac.uk }
\vspace{0.3cm}

\noindent
V.~M: Department of Mathematics, Uppsala University, Box. 480,
SE-75106, Uppsala, SWEDEN, email: {\tt mazor\symbol{64}math.uu.se}


\begin{thebibliography}{999999}
\bibitem[Au]{Au} M.~Auslander; Representation theory of Artin algebras. 
I.  Comm.  Algebra  {\bf 1} (1974), 177--268.
\bibitem[BFK]{BFK} J.~Bernstein, I.~Frenkel, M.~Khovanov;
A categorification of the Temperley-Lieb algebra and Schur quotients 
of $U(\mathfrak{sl}_2)$ via projective and Zuckerman functors.
Selecta Math. (N.S.) {\bf 5} (1999), no. 2, 199--241. 
\bibitem[Br]{Br} R.~Brauer; On algebras which are connected with the 
semisimple continuous groups. Ann. Math. (2) {\bf 38} (1937), no. 4,
857--872.
\bibitem[CPS]{CPS} E.~Cline, B.~Parshall, L.~Scott; Generic 
and $q$-rational representation theory. Publ. Res. Inst. Math. Sci. 
{\bf 35} (1999), no. 1, 31--90. 
\bibitem[CMPX]{CMPX} A.~Cox, P.~Martin, A.~Parker, C.~Xi; 
Representation theory of towers of recollement: theory, notes, and 
examples. Journal of Algebra {\bf 302} (2006), no. 1, 340--360.
\bibitem[FW]{FW} D.~Fitzgerald, K.~Wai Lau; On the partition
monoid and some related semigroups. To appear in Bull. Aust. Math. Soc.
\bibitem[GM1]{GM1} O.~Ganyushkin, V.~Mazorchuk; Factor powers of 
finite symmetric groups. (Russian) Mat. Zametki {\bf 58} (1995), 
no. 2, 176--188; translation in Math. Notes {\bf 58} (1995), 
no. 1-2, 794--802.
\bibitem[GM2]{GM2} O.~Ganyushkin, V.~Mazorchuk; On the radical of 
$\mathcal{FP}^+(S_n)$. Mat. Stud. {\bf 20} (2003), no. 1, 17--26.
\bibitem[GM3]{GM} O.~Ganyushkin, V.~Mazorchuk; Classical finite 
transformation semigroups. An introduction. Algebra and Applications, 
{\bf 9}. Springer-Verlag London, Ltd., London, 2009.
\bibitem[Gr]{Gr} C.~Grood; The rook partition algebra. J. Combin. Theory 
Ser. A {\bf 113} (2006), no. 2, 325???351.
\bibitem[HL]{HL} T.~Halverson, T.~Lewandowski; RSK insertion for set 
partitions and diagram algebras. Electron. J. Combin. {\bf 11} 
(2004/06), no. 2, Research Paper 24, 24 pp.
\bibitem[KR]{KR} K.~Kim, F.~Roush; Two-generator semigroups of 
binary relations. J. Math. Psych. {\bf 17} (1978), no. 3, 236--246.
\bibitem[Koe]{Koe} S.~Koenig; A panorama of diagram algebras. Trends 
in representation theory of algebras and related topics, 491--540, 
EMS Ser. Congr. Rep., Eur. Math. Soc., Z{\"u}rich, 2008. 
\bibitem[Ko]{Ko} J.~Konieczny; Green's equivalences in finite semigroups 
of binary relations. Semigroup Forum {\bf 48} (1994), no. 2, 235--252.
\bibitem[KM]{KM} G.~Kudryavtseva, V.~Mazorchuk; Partialization of 
categories and inverse braid-permutation monoids. Internat. 
J. Algebra Comput. {\bf 18} (2008), no. 6, 989--1017.
\bibitem[Mar1]{Mar1} P.~Martin; Temperley-Lieb algebras for nonplanar 
statistical  mechanics---the partition algebra construction.  
J. Knot Theory Ramifications  {\bf 3}  (1994),  no. 1, 51--82.
\bibitem[Mar2]{Mar2} P.~Martin; Potts models and related problems 
in statistical mechanics. Series on Advances in Statistical Mechanics, 
{\bf 5}. World Scientific Publishing Co., Inc., Teaneck, NJ, 1991.
\bibitem[Mar3]{Mar3} P.~Martin; On diagram categories, representation 
theory and Statistical Mechanics. AMS Contemporary Math {\bf 456} 
(2008) 99--136.
\bibitem[MM]{MM} P.~Martin, V.~Mazorchuk; On the representation theory 
of partial Brauer algebras. To appear.
\bibitem[Maz1]{Maz1} V.~Mazorchuk; On the structure of the Brauer 
semigroup and its partial analogue. Problems in Algebra (Gomel) 
{\bf 13} (1998), 29--45.
\bibitem[Maz2]{Maz2} V.~Mazorchuk; Endomorphisms of $\mathfrak{B}\sb n, 
\mathcal{P}\mathfrak{B}\sb n$, and $\mathfrak{C}\sb n$.  
Comm. Algebra  {\bf 30}  (2002), no. 7, 3489--3513.
\bibitem[MP]{MP} J.~Montague, R.~Plemmons; Maximal subgroups of the 
semigroup of relations. J. Algebra {\bf 13} (1969), 575--587.
\bibitem[PW]{PW} R.~Plemmons, M.~West; On the semigroup of binary 
relations. Pacific J. Math. {\bf 35} (1970), 743--753.
\bibitem[RT]{RT} N.~Reshetikhin, V.~Turaev; Ribbon graphs and their 
invariants derived from quantum groups. Comm. Math. Phys. 
{\bf 127} (1990), no. 1, 1--26.
\bibitem[Sc]{Sc} S.~Schwarz; On the semigroup of binary relations on a 
finite set. Czechoslovak Math. J. {\bf 20} (1970), 632--679. 
\bibitem[TL]{TL} H.~Temperley, E.~Lieb; Relations between the 
``percolation'' and ``colouring'' problem and other graph-theoretical 
problems associated with regular planar lattices: some exact results for 
the ``percolation'' problem. Proc. Roy. Soc. London Ser. A {\bf 322} 
(1971), no. 1549, 251--280. 
\bibitem[Tu1]{Tu} V.~Turaev; The category of oriented tangles and 
its representations. Funct. Anal. Appl. {\bf 23} (1989), no. 3, 254--255.
\bibitem[Tu2]{Tu2} V.~Turaev; Quantum invariants of knots and 3-manifolds.
Second revised edition. de Gruyter Studies in Mathematics, {\bf 18}. 
Walter de Gruyter \& Co., Berlin, 2010.
\end{thebibliography}
\end{document}